\pdfoutput=1
\documentclass[11pt]{amsart}
\newif\ifpersonal
\usepackage{amsmath,amsthm,amssymb,mathrsfs,mathtools,bm,tensor} 
\usepackage{microtype,lmodern} 
\usepackage[utf8]{inputenc} 
\usepackage[T1]{fontenc} 
\usepackage{enumitem,comment,braket,xspace,tikz-cd,xypic,csquotes} 
\usepackage[centering,vscale=0.7,hscale=0.7]{geometry}
\usepackage[hidelinks]{hyperref}

\DeclareRobustCommand{\SkipTocEntry}[5]{} 
\allowdisplaybreaks

\numberwithin{equation}{section}
\theoremstyle{plain}
\newtheorem{theorem}[equation]{Theorem}
\newtheorem{lemma}[equation]{Lemma}
\newtheorem{question}[equation]{Question}

\newtheorem*{claim*}{Claim}

\newtheorem{proposition}[equation]{Proposition}
\newtheorem{corollary}[equation]{Corollary}
\theoremstyle{definition}

\newtheorem{example}[equation]{Example}
\theoremstyle{remark}
\newtheorem{remark}[equation]{Remark}

\ifpersonal
\newcommand{\personal}[1]{\textcolor[rgb]{0,0,1}{(Personal: #1)}}
\else
\newcommand{\personal}[1]{\ignorespaces}
\fi

\providecommand{\abs}[1]{\lvert#1\rvert}

\newcommand{\bbA}{\mathbb A}

\newcommand{\bbC}{\mathbb C}

\newcommand{\fC}{\mathfrak C}

\newcommand{\fM}{\mathfrak M}

\newcommand{\cI}{\mathcal I}

\newcommand{\cO}{\mathcal O}

\makeatletter
\let\save@mathaccent\mathaccent
\newcommand*\if@single[3]{%
	\setbox0\hbox{${\mathaccent"0362{#1}}^H$}%
	\setbox2\hbox{${\mathaccent"0362{\kern0pt#1}}^H$}%
	\ifdim\ht0=\ht2 #3\else #2\fi
}
\newcommand*\rel@kern[1]{\kern#1\dimexpr\macc@kerna}
\newcommand*\widebar[1]{\@ifnextchar^{{\wide@bar{#1}{0}}}{\wide@bar{#1}{1}}}
\newcommand*\wide@bar[2]{\if@single{#1}{\wide@bar@{#1}{#2}{1}}{\wide@bar@{#1}{#2}{2}}}
\newcommand*\wide@bar@[3]{%
	\begingroup
	\def\mathaccent##1##2{%
		\let\mathaccent\save@mathaccent
		\if#32 \let\macc@nucleus\first@char \fi
		\setbox\z@\hbox{$\macc@style{\macc@nucleus}_{}$}%
		\setbox\tw@\hbox{$\macc@style{\macc@nucleus}{}_{}$}%
		\dimen@\wd\tw@
		\advance\dimen@-\wd\z@
		\divide\dimen@ 3
		\@tempdima\wd\tw@
		\advance\@tempdima-\scriptspace
		\divide\@tempdima 10
		\advance\dimen@-\@tempdima
		\ifdim\dimen@>\z@ \dimen@0pt\fi
		\rel@kern{0.6}\kern-\dimen@
		\if#31
		\overline{\rel@kern{-0.6}\kern\dimen@\macc@nucleus\rel@kern{0.4}\kern\dimen@}%
		\advance\dimen@0.4\dimexpr\macc@kerna
		\let\final@kern#2%
		\ifdim\dimen@<\z@ \let\final@kern1\fi
		\if\final@kern1 \kern-\dimen@\fi
		\else
		\overline{\rel@kern{-0.6}\kern\dimen@#1}%
		\fi
	}%
	\macc@depth\@ne
	\let\math@bgroup\@empty \let\math@egroup\macc@set@skewchar
	\mathsurround\z@ \frozen@everymath{\mathgroup\macc@group\relax}%
	\macc@set@skewchar\relax
	\let\mathaccentV\macc@nested@a
	\if#31
	\macc@nested@a\relax111{#1}%
	\else
	\def\gobble@till@marker##1\endmarker{}%
	\futurelet\first@char\gobble@till@marker#1\endmarker
	\ifcat\noexpand\first@char A\else
	\def\first@char{}%
	\fi
	\macc@nested@a\relax111{\first@char}%
	\fi
	\endgroup
}
\makeatother

\newcommand{\PD}{\mathrm{PD}}

\newcommand{\KX}{\mathit{KX}}

\newcommand{\longto}{\longrightarrow}

\newcommand{\ZZ}{\mathbb Z}
\newcommand{\QQ}{\mathbb Q}

\newcommand{\PP}{\mathbb P}

\newcommand{\BGm}{B\Gm}
\newcommand{\loc}{\mathrm{loc}}
\newcommand{\pt}{\mathrm{pt}}
\newcommand{\tev}{\widetilde{\ev}}
\newcommand{\tzeta}{\widetilde{\zeta}}

\DeclareMathOperator{\Res}{Res}

\newcommand{\bfMap}{\mathbf{Map}}

\newcommand{\vir}{\mathrm{vir}}

\newcommand{\bcM}{\widebar{\mathcal M}}

\newcommand{\Gm}{\mathbb G_{\mathrm m}}

\newcommand{\ev}{\mathrm{ev}}

\newcommand{\id}{\mathrm{id}}

\newcommand{\GW}{\mathrm{GW}}

\providecommand{\abs}[1]{\lvert#1\rvert}

\newcommand*{\longhookrightarrow}{\ensuremath{\lhook\joinrel\relbar\joinrel\rightarrow}}

\usetikzlibrary{decorations.markings} 
\tikzset{
  closed/.style = {decoration = {markings, mark = at position 0.5 with { \node[transform shape, xscale = .8, yscale=.4] {/}; } }, postaction = {decorate} },
  open/.style = {decoration = {markings, mark = at position 0.5 with { \node[transform shape, scale = .7] {$\circ$}; } }, postaction = {decorate} }
}

\DeclareMathOperator{\NE}{NE}

\DeclareMathOperator{\rank}{rank}

\DeclareMathOperator{\Spec}{Spec}

\begin{document}

\title{Gromov--Witten invariants with naive tangency conditions}
\author{Felix Janda}
\address{Felix JANDA, Department of Mathematics, University of Illinois Urbana--Champaign, Urbana, IL 61801, USA}
\email{fjanda@illinois.edu}
\author{Tony Yue Yu}
\address{Tony Yue YU, Department of Mathematics M/C 253-37, California Institute of Technology, 1200 E California Blvd, Pasadena, CA 91125, USA}
\email{yuyuetony@gmail.com}
\date{October 19, 2023}

\makeatletter
\@namedef{subjclassname@2020}{%
	\textup{2020} Mathematics Subject Classification}
\makeatother

\begin{abstract}
  We introduce Gromov--Witten invariants with naive tangency conditions at the marked points of the source curve.
  We then establish an explicit formula which expresses Gromov--Witten invariants with naive tangency conditions in terms of descendent Gromov--Witten invariants.
  Several examples of genus zero Gromov--Witten invariants with naive tangencies are computed in the case of curves and surfaces.
  In particular, the counts of rational curves naively tangent to an anticanonical divisor on a del Pezzo surface are studied, and via mirror symmetry, we obtain a relation to the local Gromov--Witten invariants.
\end{abstract}

\maketitle

\tableofcontents

\section{Introduction}

It is a classical problem in algebraic geometry to count curves with given genus and homology class in a smooth complex projective variety.
We also impose incidence conditions, e.g.\ the curves should pass through given points or given subvarieties.
Such counts are usually not well-defined, i.e.\ they can be infinite, or depend on the position of the subvarieties for the incidence conditions.
This is related to compactness and transversality in enumerative geometry.
We now have a variety of enumerative theories providing well-defined counts via different compactified moduli spaces (see \cite{Pandharipande_13/2_ways}).
The theory of Gromov--Witten invariants is the most general, in the sense of being defined in all dimensions.
Furthermore, they constitute the key ingredients in the mathematical formulation of mirror symmetry (see \cite{Hori_Mirror_symmetry}).

It is natural to extend the incidence conditions in Gromov--Witten theory to tangency conditions, i.e.\ requiring the curves to be tangent to given subvarieties, rather than merely passing through them.
Examples include the beautiful relationship between Gromov--Witten invariants and classical Hurwitz numbers counting ramified covers of curves (see \cite{Okounkov_Gromov-Witten_theory_Hurwitz_theory_and_completed_cycles}), the counts of $\bbA^1$-curves with tangencies giving rise to scattering diagrams (see \cite{Gross_The_tropical_vertex,Gross_Mirror_symmetry_for_log_Calabi-Yau_surfaces_I_published}), and the degeneration formula for Gromov--Witten invariants (see \cite{Li_A_degeneration_formula}).
The above curve counts with tangencies are all formulated in the language of relative Gromov--Witten theory (see \cite{Li_Stable_morphisms,Ionel_Relative_Gromov-Witten_invariants}), which was later vastly generalized to the theory of log Gromov--Witten invariants (see \cite{Gross_Logarithmic_Gromov-Witten_invariants, Ch14, AbCh14}).

The idea of curve counting with tangencies is also essential in the non-archimedean mirror symmetry program, for the construction of mirror structure constants and theta functions for log Calabi--Yau varieties (see \cite{Yu_Enumeration_of_holomorphic_cylinders_I,Yu_Enumeration_of_holomorphic_cylinders_II,Keel_Yu_The_Frobenius}).
One distinctive advantage of the non-archimedean approach is the freeness of non-archimedean analytic curves in the generic fiber, versus the log curves in the special fiber.
Thanks to this, we were able to impose the tangency conditions in a naive way using jets, without the need of target expansions or sophisticated log structures.

The theory of non-archimedean Gromov--Witten invariants, as well as its K-theoretic version, were developed in \cite{Porta_Yu_Non-archimedean_quantum_K-invariants,Porta_Yu_Non-archimedean_Gromov-Witten_invariants}, based on derived non-archimedean geometry \cite{Porta_Yu_Derived_non-archimedean_analytic_spaces,Porta_Yu_Representability_theorem,Porta_Yu_Derived_Hom_spaces}.
The flexibility of the derived approach allowed us to include naive tangency conditions with little extra effort.
To the best our knowledge, despite their simplicity, such invariants were not considered in general in the literature.

\medskip

In this paper, we set up such Gromov--Witten invariants with naive tangency conditions in the algebraic setting, using the classical language of virtual pull-backs (see~\cite{Behrend_Intrinsic_normal_cone, Manolache_Virtual_pull-backs}).
We establish an explicit formula which expresses Gromov--Witten invariants with naive tangency conditions in terms of classical descendent Gromov--Witten invariants.

\medskip

More explicitly, let $X$ be a smooth complex projective variety.
We consider the moduli stack $\bcM_{g,n}(X,\beta)$ of genus $g$, $n$-pointed stable maps into $X$ of homology class $\beta$.
For every marked point $p_i$, we fix an lci subvariety $Z_i \subset X$ (which is not necessarily a divisor) and a positive integer $m_i$.
In Section \ref{sec:naive_tangencies_definition}, we introduce a generalized evaluation map
\[ \ev_i^{m_i}\colon \bcM_{g,n}(X,\beta)\longrightarrow X_i^{m_i},\]
taking into account the $(m_i - 1)$-jets at the marked point $p_i$.
The subvariety $Z_i \subset X$ induces a derived lci inclusion $\zeta_i \colon Z_i^{m_i} \to X_i^{m_i}$ of jet spaces over the stack of prestable curves.
Taking the Gysin pushforward $\zeta_{i!}(1)$, and capping the pullbacks $(\ev_i^{m_i})^*\zeta_{i!}(1)$ with the usual virtual fundamental class $[\bcM_{g,n}(X,\beta)]^\vir$, we obtain Gromov--Witten counts of curves whose marked point $p_i$ meets $Z_i$ with tangency order at least $m_i$.
We denote the resulting numbers by $\langle m_1 Z_1, \dots, m_n Z_n \rangle_{g, \beta}^X$.

We establish the following relation between Gromov--Witten invariants with naive tangency conditions and descendent Gromov--Witten invariants.
\begin{theorem}[see Theorem {\ref{thm:naive_tangencies_to_descendents}}] \label{thm:intro:naive_tangencies_to_descendents}
  We have
  \begin{equation*}
    (\ev_i^{m_i})^* \zeta_{i!}(1)
    = \ev_i^* \iota_{i!} \prod_{k = 1}^{m_i - 1} \sum_{j = 0}^{r_i} (k\psi_i)^{r_i - j} c_j(N_i),
  \end{equation*}
  where $\iota_i$ denotes the inclusion $Z_i \hookrightarrow X$, $N_i$ is the normal bundle of $Z_i$ in $X$, and $r_i = \rank(N_i)$.
  More precisely, we view the right-hand side as a polynomial in the cotangent line class $\psi_i$ with coefficients in $A^*(Z_i)$, and apply to each coefficient the Gysin pushforward $\iota_{i!}$ and the pullback $\ev_i^*$ by the ordinary evaluation map.
\end{theorem}

The above theorem enables us to convert a single naive tangency condition to descendent insertions.
Applying this relation to each marked point, we obtain a formula expressing Gromov--Witten invariants with naive tangency conditions in terms of descendent Gromov--Witten invariants.

In the special case where $Z_i$ is a Cartier divisor (see Corollary~\ref{cor:divisor}), we obtain
\[
  (\ev_i^{m_i})^* \zeta_{i!}(1)
  = \prod_{k = 0}^{m_i - 1} (k\psi_i + \ev_i^* c_1(L))
  = \sum_{k = 0}^{m_i} \genfrac{[}{]}{0pt}{1}{m_i}{k} \psi_i^{m_i - k} (\ev_i^* c_1(L))^k,
\]
where $\genfrac{[}{]}{0pt}{1}{m_i}{k}$ denotes the unsigned Stirling numbers of the first kind.
In another special case where $Z_i$ is a point in $X$ (see Example~\ref{ex:point}),
assuming $X$ has dimension $d$, we obtain
\[
  (\ev_i^{m_i})^* \zeta_{i!}(1)
  = ((m_i - 1)! \psi_i^{m_i - 1})^d \ev_i^* \PD^{-1} [\pt].
\]

Theorem \ref{thm:intro:naive_tangencies_to_descendents} may be viewed as an analog of Maulik--Pandharipande's work \cite{MaPa06}, which connects relative Gromov--Witten invariants to absolute descendent Gromov--Witten invariants via a nonexplicit recursive relation.
In contrast, the relationship between Gromov--Witten invariants with naive tangencies and descendent Gromov--Witten invariants is direct and explicit.
Thus we propose the following natural question, which was implicitly considered in Gathmann's work \cite{Ga02}.

\begin{question}
  Is there a formula expressing relative Gromov--Witten invariants in terms of Gromov--Witten invariants with naive tangency conditions (and vice versa)?
\end{question}

Here is the basic idea of the proof of Theorem~\ref{thm:intro:naive_tangencies_to_descendents}:
The evaluation space $X^{m_i}$ in the construction of naive tangencies is defined relative to the moduli stack of prestable curves, keeping track of the geometry of the source curves.
In Section~\ref{sec:alternative}, we give another construction of naive tangencies using the alternative evaluation space $[J^{m_i-1} X / \Gm]$, which is the stack quotient of the jet scheme $J^{m_i-1} X$ modulo the scaling action of $\Gm$.
In Section~\ref{sec:relation_to_descendent}, we compare the two constructions, and then apply the virtual localization formula on the stacky evaluation spaces to obtain the explicit formula in relation to the descendent invariants.

\medskip

In Section~\ref{sec:examples}, we compute several examples of genus zero Gromov--Witten invariants with naive tangencies in the case of curves and surfaces.
More specifically, in Section~\ref{sec:curves}, we consider the case where the target is a curve.
Then Gromov--Witten invariants with naive tangencies are in fact equal to the descendent invariants, and combined with the Gromov--Witten/Hurwitz correspondence
\cite{Okounkov_Gromov-Witten_theory_Hurwitz_theory_and_completed_cycles},
we find the formula
\[ \bigg\langle\prod_{i=1}^n k_i q_i\bigg\rangle_d^{\bullet X} = \sum_{\abs{\lambda}=d} \bigg(\frac{\dim\lambda}{d!}\bigg)^{2-2g(X)} \prod_{i=1}^n \frac{\mathbf p_{k_i}(\lambda)}{k_i}, \]
where the left-hand side denotes the disconnected Gromov--Witten invariants with naive tangencies, and $\mathbf p_{k_i}$ is the shifted symmetric power sum in Hurwitz theory.

We study the case of surfaces in Sections~\ref{sec:lines_tangent_to_conic}--\ref{sec:arbitrary-del-Pezzo}.
We begin with simple examples where the invariants are enumerative.
Then, in Section~\ref{sec:gathmann}, we compare to Gathmann's computation \cite{Ga05} of plane conics that are tangent to a degree $d$ plane curve at $5$ points, and find that the corresponding Gromov--Witten invariants with naive tangencies are given by
\begin{equation*}
  n_d^{\GW} = \frac 1{5!} d^8 (d^2 + 5d + 20),
\end{equation*}
which is different from the enumerative count.

Finally, in Section~\ref{sec:arbitrary-del-Pezzo}, we study the counts of rational curves naively tangent to an anticanonical divisor $D$ on a del Pezzo surface $X$, and establish a relation to the Gromov-Witten invariants of the corresponding local del Pezzo surface $\KX$, via mirror symmetry.
More specifically, let $F(q)$ denote the generating series of genus zero, one-pointed Gromov--Witten invariants of $X$ with a naive tangency of maximal order along $D$, and let $F^\loc(Q)$ denote an analogous generating series of ordinary Gromov--Witten invariants of $\KX$.
We find:
\begin{theorem}[see Theorem~\ref{prop:naive-to-log}]
  We have
  \begin{equation*}
    F(q)|_{q^\beta \mapsto (-1)^{\beta \cdot D} q^\beta} - \frac{d_X}2 I_1^2(q)= -F^\loc(Q)
  \end{equation*}
  under the mirror map
  \begin{equation*}
    Q^\beta = q^\beta \cdot \exp\big((\beta \cdot D) I_1(q)\big),
  \end{equation*}
  where $I_1$ is obtained from a hypergeometric modification of Givental's $J$-function of $X$.
\end{theorem}
Combining with the results of \cite{GGR19}, we can compare $F(q)$ to the analogous generating series $F^{\log}(Q)$ of log (or relative) Gromov--Witten invariants.
\begin{corollary}[see Corollary~\ref{cor:naive-to-log}]
  Under the same mirror transformation, we have
  \begin{equation*}
    F(q)|_{q^\beta \mapsto (-1)^{\beta \cdot D} q^\beta} - \frac{d_X}2 I_1^2(q) = F^{\log}(Q)|_{Q^\beta \mapsto (-1)^{\beta \cdot D} Q^\beta}
  \end{equation*}
\end{corollary}
In the case of $X = \PP^2$, we obtain an explicit formula
\begin{equation*}
  F(q) = 3 \sum_{d = 0}^\infty q^d \frac{(3d - 1)!}{d!^3} \sum_{i = d + 1}^{3d - 1} \frac 1i
\end{equation*}
for the generating series of Gromov--Witten invariants with naive tangencies (see Example~\ref{ex:local-P2}).

\addtocontents{toc}{\SkipTocEntry}
\subsection*{Acknowledgements}
We are grateful to M.~Kontsevich who initially suggested a possible relationship between naive tangencies and descendent invariants.
Discussions with Q.~Chen, S.~Guo, D.~Maulik, R.~Pandharipande and M.~Porta were very helpful.
Special thanks to A.~Polishchuk and Y.~Shen for organizing the workshop ``Topics in Enumerative Geometry'' at the University of Oregon in May 2022, during which we first exchanged ideas.
F.~Janda was partially supported by NSF grants DMS-2054830 and DMS-2239320.
T.Y.~Yu was partially supported by NSF grants DMS-2302095 and DMS-2245099.

\section{Review of Chow groups of algebraic stacks} \label{sec:Chow_groups}

Throughout the paper, we work over the field $\bbC$ of complex numbers.
We make use of rational Chow groups $A_*$ of algebraic stacks in the sense of Kresch \cite{Kr99}, and the generalization by Bae--Schmitt \cite[Appendix~A]{BaSc22} from stacks of finite type to stacks locally of finite type\footnote{The generalization to stacks locally of finite type is only needed in Section~\ref{sec:naive_tangencies_definition}.
It is possible to circumvent it by restricting to certain open substacks of finite type, but this would complicate the notation.}.
For a quotient stack $X = [Y / G]$, Kresch's Chow groups $A_*(X)$ agree with the equivariant Chow groups $A_*^G(Y)$ of Edidin--Graham \cite{EG98}.
We will also use the operational Chow groups $A^*$ from \cite[Appendix~C]{BaSc22}, which we recall here:

Let $X$ be an algebraic stack locally of finite type that admits a stratification by global quotient stacks in the sense of Kresch \cite{Kr99}.
An operational class $c \in A^*(X)$ is a collection of homomorphisms
\begin{equation*}
  c(g)\colon A_*(B) \to A_*(B)
\end{equation*}
for all morphisms $B \to X$ where $B$ is an algebraic stack of finite type over $\bbC$, stratified by global quotient stacks.
The collection is required to be compatible with representable proper pushforward,
flat pullback and refined Gysin pullback along representable lci morphisms.
By \cite[Corollary~C.7]{BaSc22}, the group $A^*(X)$ forms a graded ring with unit $1 \in A^0(X)$ defined by the identity operation, with a grading such that elements of $A^k(X)$ decrease degree by $k$.
Examples of elements in $A^k(X)$ include the Chern classes $c_k(E)$ of any vector bundle $E$ on $X$.

We next define a Gysin pushforward operation for morphisms admitting a relative perfect obstruction theory.
Let $X$ and $Y$ be algebraic stacks locally of finite type over $k$ which are stratified by global quotient stacks, and $f\colon X \to Y$ be a proper morphism equipped with a relative perfect obstruction theory in the sense of Behrend--Fantechi \cite{Behrend_Intrinsic_normal_cone}.

For any cartesian diagram
\begin{equation*}
  \xymatrix{
    X' \ar[r]^{f'} \ar[d] & Y' \ar[d] \\
    X \ar[r]_f & Y
  }
\end{equation*}
with $X'$ and $Y'$ admitting stratifications by global quotient stacks, we obtain a virtual pullback $f^!\colon A_*(Y') \to A_*(X')$ in the sense of Manolache \cite{Manolache_Virtual_pull-backs}.

Now, given a class $\alpha \in A^*(X)$, represented by operations $\alpha\colon A_*(X') \to A_*(X')$, we define the \emph{Gysin pushforward} $f_!\alpha \in A^*(Y)$ via the operations
\begin{equation*}
  f_! \alpha := f_*' \circ \alpha \circ f^! \colon A_*(Y') \to A_*(Y').
\end{equation*}
It follows from the properties of virtual pullbacks \cite[Theorems~4.1,
4.3]{Manolache_Virtual_pull-backs} that this is an operational class (see also \cite[Theorem~4.4]{Manolache_Virtual_pull-backs}).
The Gysin pushforward is a morphism of vector spaces that preserves neither the grading nor the ring structure.
\begin{remark}
  When $f$ is an lci morphism of schemes, the Gysin pushforward $f_!$ associated to the canonical obstruction theory of $f$ agrees with the Gysin pushforward of Fulton \cite[p.~328, (G\textsubscript{2})]{Fulton_Intersection_theory}.
\end{remark}

By \cite[Lemma~C.6]{BaSc22}, if $X$ is smooth equidimensional of dimension $n$, there is a fundamental class $[X] \in A_n(X)$, and a Poincar\'e duality isomorphism
\begin{equation*}
  \PD\colon A^*(X) \to A_*(X), \quad \alpha \mapsto \alpha \cap [X].
\end{equation*}

In the following, we will only consider Gysin pushforwards of the unit $1$ and Chern classes.
In these cases, $f_!(\alpha)$ can be computed in terms of the virtual fundamental class.

\begin{lemma}
  \label{lem:explicit-Gysin-pushforward}
  Let $X$ and $Y$ be algebraic stacks locally of finite type over $k$ which are stratified by global quotient stacks, and let $f\colon X \to Y$ be a proper morphism equipped with a relative perfect obstruction theory.
  Furthermore, assume that $Y$ is smooth and equidimensional, so that we obtain a virtual fundamental class $[X]^\vir = f^! [Y]$ via virtual pullback.
  Then,
  \begin{equation*}
    f_!(1) = \PD^{-1} f_* [X]^\vir,
  \end{equation*}
  where $1 \in A^*(X)$ is the unit.
  More generally, for any vector bundle $E$ on $X$, we have
  \begin{equation*}
    f_!(c_k(E)) = \PD^{-1} f_* \big(c_k(E) \cap [X]^\vir\big),
  \end{equation*}
\end{lemma}
\begin{proof}
  We compute
  \begin{equation*}
    \PD f_!\big(c_k(E)\big)
    = f_!\big(c_k(E)\big) \cap [Y]
    = f_* \big(c_k(E) \cap f^! [Y]\big)
    = f_* \big(c_k(E) \cap [X]^\vir\big).
  \end{equation*}
  Applying $\PD^{-1}$ to both sides yields the second identity.
  The first identity follows by taking $k = 0$.
\end{proof}

\section{Gromov--Witten invariants with naive tangencies} \label{sec:naive_tangencies_definition}

In this section, we introduce a generalization of the evaluation maps taking into account the jets at the marked points, and then we define Gromov--Witten invariants with naive tangency conditions.

Given a smooth projective variety $X$, natural numbers $g, n$, and a homology class $\beta \in H_2(X; \ZZ)$, let $\bcM_{g, n}(X, \beta)$ be the moduli stack of genus $g$, $n$-pointed stable maps into $X$ of class $\beta$,
$\fC_{g, n}$ the universal curve, and $\fM_{g, n}$ the moduli stack of genus $g$, $n$-pointed pre-stable curves.
We have a natural map $\bcM_{g, n}(X, \beta) \to \fM_{g, n}$ taking every stable map to its source curve.

To define Gromov--Witten invariants with naive tangencies, we generalize the evaluation map
\[\ev_i\colon\bcM_{g,n}(X,\beta)\longto X\]
to take into account the jets at the $i$-th marked point.
Fix a positive integer $m_i$.
Let $s_i\colon\fM_{g,n}\to\fC_{g,n}$ be the $i$-th section of the universal curve, $\cI_i$ the ideal sheaf on $\fC_{g,n}$ cutting out the image of $s_i$, and $(\fC_{g,n})_{(s_i^{m_i})}$ the closed substack given by the $m_i$-th power of the ideal $\cI_i$.
Let
\[\bfMap_{\fM_{g,n}}\bigl(\fC_{g,n}, X\times\fM_{g,n}\bigr)\]
denote the derived mapping stack (see \cite[\S3.3]{DAG-XIV} and \cite{Porta_Yu_Derived_Hom_spaces}),
and let $\ev_i^{m_i}$ denote the composition of the inclusion
\[\bcM_{g,n}(X,\beta) \longhookrightarrow \bfMap_{\fM_{g,n}}\bigl(\fC_{g,n}, X\times\fM_{g,n}\bigr)\]
with the restriction
\[\bfMap_{\fM_{g,n}}\bigl(\fC_{g,n}, X\times\fM_{g,n}\bigr) \longto \bfMap_{\fM_{g,n}}\Bigl((\fC_{g,n})_{(s_i^{m_i})}, X\times \fM_{g,n}\Bigr)\eqqcolon X_i^{m_i}.\]
We call
\[\ev_i^{m_i}\colon\bcM_{g,n}(X,\beta)\longrightarrow X_i^{m_i}\]
the \emph{evaluation map of the $i$-th marked point of order $m_i$}.

Given any closed subscheme $Z_i\subset X$ such that the inclusion is lci, let
\[Z_i^{m_i}\coloneqq\bfMap_{\fM_{g,n}}\Bigl((\fC_{g,n})_{(s_i^{m_i})}, Z_i\times \fM_{g,n}\Bigr).\]
By \cite[Lemma 9.1]{Porta_Yu_Non-archimedean_quantum_K-invariants}, the derived stack $X_i^{m_i}$ is smooth over $\fM_{g,n}$, in particular underived; and the map $\zeta_i\colon Z_i^{m_i}\to X_i^{m_i}$ is derived lci.
Therefore, it follows from \cite[\S1]{Schurg_Dervied_algebraic_geometry_determinants} that $\zeta_i$ admits a relative perfect obstruction theory.
Hence, following Section \ref{sec:Chow_groups}, we have the Gysin pushforward map
\[\zeta_{i!}\colon A^*(Z_i^{m_i})\longto A^*(X_i^{m_i}) .\]
By Lemma~\ref{lem:explicit-Gysin-pushforward}, $\zeta_{i!}(1)\in A^*(X_i^{m_i})$ is Poincaré dual to the pushforward by $\zeta_{i*}$ of the virtual fundamental class $[Z_i^{m_i}]^\vir \in A_*(Z_i^{m_i})$.
Note that except in the case $(g, n) = (1, 0)$, which is irrelevant to us, by \cite[Proposition~4.5.5(ii)]{Kr99} and \cite[\S~3.1]{BaSc22}, the mapping stacks are stratified by quotient stacks, and thus have a well-behaved theory of Chow groups.

We then define the \emph{Gromov--Witten invariant of genus $g$ and class $\beta$, with naive tangencies of order $m_i$ along $Z_i$} as
\begin{align*}
  \big\langle m_1 Z_1, \dots, m_n Z_n \big\rangle_{g, \beta}^X
  &=\int_{[\bcM_{g, n}(X, \beta)]^\vir} \prod_{i = 1}^n (\ev_i^{m_i})^* \zeta_{i!}(1) \\
  &= \deg\left( \bigg(\prod_{i = 1}^n (\ev_i^{m_i})^* \zeta_{i!}(1)\bigg) \cap [\bcM_{g, n}(X, \beta)]^\vir\right),
\end{align*}
where $[\bcM_{g, n}(X, \beta)]^\vir \in A_*(\bcM_{g, n})$ denotes the virtual fundamental class.

More generally, for the $i$-th marked point, we may consider a finite collection of lci subvarieties $Z_{i, j} \subset X$, together with positive integers $m_{i, j}$, for $j = 1, \dots, k_i$.
Then we can ask for the curve to be tangent to each $Z_{i, j}$ with order at least $m_{i, j}$.
The resulting Gromov--Witten invariant with naive tangencies under this generalization is defined as
\begin{equation*}
  \bigg\langle \bigcap_{j=1}^{k_1} m_{1,j} Z_{1,j}, \dots, \bigcap_{j=1}^{k_n} m_{n,j} Z_{n,j} \bigg\rangle_{g, \beta}^X =
  \int_{[\bcM_{g, n}(X, \beta)]^\vir} \prod_{i = 1}^n \prod_{j = 1}^{k_i} (\ev_i^{m_{i, j}})^* \zeta_{i, j!}(1)
\end{equation*}

\begin{remark}
  \label{rmk:log}
  Unlike relative Gromov--Witten theory, naive tangencies may be imposed with respect to subvarieties $Z_i$ that are not divisors.
  In the case that $Z_i \subset X$ is a transverse intersection of smooth hypersurfaces $H_{i,j} \subset X$, imposing a naive tangency of order $k$ along $Z_i$ at the $i$-th marked point is the same as imposing a naive tangency of order $k$ along every $H_{i,j}$ at the same marked point.
  Thus, a naive tangency condition along $Z_i$ may be related to log Gromov--Witten invariants with respect to the snc divisor $\bigcup_j H_{i,j}$.
  This is also a motivation behind the above generalization to the collections $Z_{i,j}$.
\end{remark}

\begin{remark}
  More generally, we may also allow $Z_{i,j}$ to be derived schemes, and consider derived lci inclusions $Z_{i,j} \to X$.
  For instance, $Z_{i,j}$ can be the vanishing locus of a non-transverse section of a vector bundle.
  Furthermore, pushing-forward to the Deligne--Mumford moduli stack of curves $\bcM_{g,n}$ instead of integrating, we obtain \emph{Gromov--Witten classes with naive tangencies}.
\end{remark}

\section{Alternative evaluation space}
\label{sec:alternative}

In this section, we give an equivalent construction of Gromov--Witten invariants with naive tangencies using an alternative smaller evaluation space that still retains sufficient information about tangency orders.
We continue the setting of Section~\ref{sec:naive_tangencies_definition}.

Let $\BGm$ be the classifying stack of the multiplicative group $\Gm$, and $\varphi_i\colon \fM_{g, n} \to \BGm$ the morphisms given by the tangent line bundle $L_i$ at each marked point.
Let $\Lambda$ denote the $1$-dimensional representation of $\Gm$ with weight $1$.
Then we have $\varphi_i^* \Lambda = L_i$.

\begin{lemma} \label{lem:varphi_cartesian}
  There is a cartesian diagram
  \begin{equation*}
    \xymatrix{
      (\fC_{g, n})_{(s_i^{m_i})} \ar[r] \ar[d] & [\Spec(k[t]/(t^{m_i})) / \Gm] \ar[d] \\
      \fM_{g, n} \ar[r]^{\varphi_i} & \BGm,
    }
  \end{equation*}
  where $\Gm$ acts on $\Spec(k[t]/(t^{m_i}))$ by scaling $\lambda \cdot t = \lambda t$, and the upper horizontal arrow is given by $t \mapsto s_i$.
\end{lemma}
\begin{proof}
  The $i$-th marking of the universal curve gives a Cartier divisor $D_i \subset \fC_{g, n}$.
  The residue morphism induces a canonical trivialization $\omega(D_i)|_{D_i} \simeq \cO_{\fM_{g, n}}$, where $\omega$ is the relative dualizing sheaf.
  So we obtain $\cO(D_i)|_{D_i} \simeq L_i$.
  Let $F$ be the fiber product fitting into the cartesian diagram
  \begin{equation*}
    \xymatrix{
      F \ar[r] \ar[d] & [\Spec(k[t]/(t^{m_i})) / \Gm] \ar[d] \\
      \fM_{g, n} \ar[r]^{\varphi_i} & \BGm.
    }
  \end{equation*}
  Then, $F$ is the $(m_i-1)$-th infinitesimal thickening of the $0$-section of $L_i$.
  We can check locally that the map $(\fC_{g, n})_{(s_i^{m_i})} \to F$ is an isomorphism.
\end{proof}

Let $J^{m_i-1} X = \bfMap\big(\Spec(k[t]/(t^{m_i})), X\big)$ be the space of $(m_i - 1)$-th jets of $X$.
It is equipped with a $\Gm$-action induced by the above $\Gm$-action on $\Spec(k[t]/(t^{m_i}))$.
By Lemma~\ref{lem:varphi_cartesian}, we obtain natural maps
\begin{align*}
  X^{m_i} & \xrightarrow{\: \sim\: }\bfMap_{\fM_{g,n}} \Bigl((\fC_{g,n})_{(s_i^{m_i})}, X\times \fM_{g,n}\Bigr) \\
  & \xrightarrow{\: \sim\: } \bfMap_{B\Gm} \big([\Spec(k[t]/(t^{m_i})) / \Gm], X \times B\Gm \big) \times_{B\Gm} \fM_{g,n} \\
  & \longto \bfMap_{B\Gm} \big([\Spec(k[t]/(t^{m_i})) / \Gm], X \times B\Gm \big) \\
  & \xrightarrow{\: \sim\: } [J^{m_i-1} X / \Gm].
\end{align*}
The same constructions work for any $Z_i \subset X$.

\begin{corollary}
  \label{cor:derived-pullback}
  There is a commutative diagram
  \begin{equation*}
    \xymatrix{
      & X^{m_i} \ar[rr] \ar[dd]|\hole && [J^{m_i-1}X / \Gm] \ar[dd] \\
      Z^{m_i} \ar[ru]^{\zeta_i} \ar[rd] \ar[rr] && [J^{m_i-1} Z_i / \Gm] \ar[ru]^{\tzeta_i} \ar[rd] & \\
      & \fM_{g, n} \ar[rr]^{\varphi_i} && \BGm,
    }
  \end{equation*}
  where the square and parallelograms are derived pullbacks.
\end{corollary}

Define the evaluation map
\begin{equation*}
  \tev_i^{m_i}\colon \bcM_{g, n}(X, \beta) \longto [J^{m_i-1} X / \Gm]
\end{equation*}
as the composition
\[
  \bcM_{g, n}(X, \beta) \xrightarrow{\ev_i^{m_i}} X^{m_i} \longto  [J^{m_i-1} X / \Gm].
\]

\begin{lemma} \label{lem:alternative_evaluation}
  The inclusion map
  $\tzeta_i\colon [J^{m_i-1} Z_i / \Gm] \to [J^{m_i-1} X / \Gm]$
  admits a relative perfect obstruction theory.
  We have
  \begin{equation}
    \label{eq:compare-insertion}
    (\ev_i^{m_i})^* \zeta_{i!}(1) = (\tev_i^{m_i})^* \tzeta_{i!}(1).
  \end{equation}
\end{lemma}
\begin{proof}
  Since the inclusion $Z_i \subset X$ is lci by assumption, the same proof of \cite[Lemma 9.1]{Porta_Yu_Non-archimedean_quantum_K-invariants} implies that the map $\tzeta_i\colon [J^{m_i-1} Z_i / \Gm] \to [J^{m_i-1} X / \Gm]$ is derived lci.
  It follows from \cite[\S~1]{Schurg_Dervied_algebraic_geometry_determinants} that $\tzeta_i$ admits a relative perfect obstruction theory.

  Furthermore, since the upper parallelogram in Corollary~\ref{cor:derived-pullback} is a derived pullback, the perfect obstruction theory for $\zeta_i$ is the pullback of the one for $\tzeta_i$.
  It follows from \cite[Theorem~4.1(iii)]{Manolache_Virtual_pull-backs} that
  \begin{equation*}
    \zeta_{i!}(1) = \varphi_i^* \tzeta_{i!}(1).
  \end{equation*}
  This implies the identity \eqref{eq:compare-insertion}, completing the proof.
\end{proof}

It follows from Lemma~\ref{lem:alternative_evaluation} that the alternative evaluation maps $\tev_i^{m_i}$ give rise to the same Gromov--Witten invariants with naive tangencies as defined in Section~\ref{sec:naive_tangencies_definition}:

\begin{proposition}
  We have
  \begin{align*}
    \big\langle m_1 Z_1, \dots, m_n Z_n \big\rangle_{g, \beta}^X
    &=\int_{[\bcM_{g, n}(X, \beta)]^\vir} \prod_{i = 1}^n (\ev_i^{m_i})^* \zeta_{i!}(1) \\
    &= \int_{[\bcM_{g, n}(X, \beta)]^\vir} \prod_{i = 1}^n (\tev_i^{m_i})^* \tzeta_{i!}(1).
  \end{align*}
\end{proposition}

\section{Relation to descendent Gromov--Witten invariants}
\label{sec:relation_to_descendent}

We now proceed to relate naive tangency conditions to descendent classes in Gromov--Witten theory.
The idea is to apply virtual localization to the $\Gm$-action on the jet-scheme $J^{m_i-1} X$ introduced in Section~\ref{sec:alternative}.

The localization formula takes place in the $\Gm$-equivariant Chow group of $J^{m_i-1} X$, which as mentioned before, agrees with $A_*([J^{m_i-1} X / \Gm])$.
Furthermore, the fixed locus of the $\Gm$-action is $J^0 X = X$ viewed as the vanishing locus of the jet bundle.
The Chow ring of $J^{m_i-1} X$ localizes to the fixed locus in the following sense:
\begin{lemma}
  \label{lem:cohomology-localizes}
  There is a factorization
  \begin{equation*}
    \xymatrix{
      A^*([J^{m_i-1} X / \Gm])  \ar[r]^{(\tev_i^{m_i})^*} & A^*(\bcM_{g, n}(X, \beta)) \\
      A^*(X \times \BGm) \ar[u]^\simeq \ar[ru]_{(\ev_i \times \varphi_i)^*}.
    }
  \end{equation*}
  Furthermore, there is an isomorphism $A^*(X \times \BGm) \simeq A^*(X)[t]$ such that $\varphi_i^*(t)$ is equal to the $i$-th tangent line class.
\end{lemma}
\begin{proof}
  By \cite[Corollary 2.11]{Ein_Jet_schemes_and_singularities}, the jet bundle $J^{m_i-1} X \to X$ has the structure of an affine bundle in the sense that it is locally of the form $X \times \mathbb A^{(m_i-1) \cdot \dim X} \to X$.

  Unwinding the definition of the equivariant Chow groups, and note that both $X \times \BGm$ and $[J^{m_i-1} X / \Gm]$ are smooth, the arrow $A^*(X \times \BGm) \to A^*([J^{m_i-1} X / \Gm])$ can be constructed via morphisms
  \begin{equation} \label{eq:Chow}
    A_*\big(X \times_{\Gm} (\mathbb A^n \setminus 0)\big) \longto A_*\big(J^{m_i-1} X \times_{\Gm} (\mathbb A^n \setminus 0)\big),
  \end{equation}
  where $\times_{\Gm}$ denotes the mixed construction with respect to the scaling $\Gm$-action on $\mathbb A^n \setminus 0$.
  We may check that $J^{m_i-1} X \times_{\Gm} (\mathbb A^n \setminus 0) \to X \times_{\Gm}
  (\mathbb A^n \setminus 0)$ also has the structure of an affine morphism, thus by \cite[Lemma~2.2]{To14}, the morphisms in \eqref{eq:Chow} constitute a compatible family of isomorphisms.
  This proves that the vertical map is an isomorphism.
  The second claim follows from \cite[Lemma~2.12 and Theorem~2.10]{To14}.
\end{proof}

This lemma immediately gives the following qualitative result:

\begin{corollary}
  All Gromov--Witten invariants with naive tangencies may be expressed in terms of ordinary descendent Gromov--Witten invariants.
\end{corollary}

We work out the precise relationship in the following theorem.

\begin{theorem} \label{thm:naive_tangencies_to_descendents}
  We have
  \begin{align*}
    (\ev_i^{m_i})^* \zeta_{i!}(1)
    =& (\ev_i \times \varphi_i)^* (\iota_i \times \id_{B\Gm})_! \prod_{k = 1}^{m_i - 1} e\big(N_i \boxtimes \Lambda^{\otimes(-k)}\big) \\
    =& \ev_i^* \iota_{i!} \prod_{k = 1}^{m_i - 1} \sum_{j = 0}^{r_i} (k\psi_i)^{r_i - j} c_j(N_i),
  \end{align*}
  where $\iota_i$ denotes the inclusion $\iota_i\colon Z_i \hookrightarrow X$,
  $N_i$ denotes the normal bundle of $Z_i$ in $X$, $r_i = \rank(N_i)$, and $\Lambda$ is the $1$-dimensional representation of $\Gm$ with weight $1$.
  For the second expression, we view it as a polynomial in the cotangent line class $\psi_i$ with coefficients in $A^*(Z_i)$, and apply to each coefficient the Gysin pushforward $\iota_{i!}$ and the pullback $\ev_i^*$.
\end{theorem}
\begin{proof}
  By Lemma~\ref{lem:alternative_evaluation}, we have
  \begin{equation*}
    (\ev_i^{m_i})^* \zeta_{i!}(1)
    = (\tev_i^{m_i})^* \tzeta_{i!}(1) .
  \end{equation*}
  We rewrite
  \begin{equation}
    \label{eq:tzeta-Gysin-PD}
    \tzeta_{i!}(1) = \PD^{-1} \tzeta_{i*} [J^{m_i-1} Z_i]^{\vir, \Gm}
  \end{equation}
  using Lemma~\ref{lem:explicit-Gysin-pushforward}.
  Here, $[J^{m_i-1} Z_i]^{\vir, \Gm} \in A_*([J^{m_i-1} Z_i / \Gm])$ is the $\Gm$-equivariant virtual fundamental class of $J^{m_i-1} Z_i$.
  We apply the virtual localization formula (see Graber--Pandharipande \cite{GrPa99}) to compute this virtual class.
  Note that the composition $[J^{m_i-1} Z_i/\Gm] \to [J^{m_i-1} X / \Gm] \to B\Gm$ is derived lci,
  thus the truncation of $J^{m_i-1} Z_i$ is endowed with the structure of a $\Gm$-equivariant perfect obstruction theory.
  It has a global $\Gm$-equivariant resolution
  \begin{equation*}
    \pi_* f^* TX|_{Z_i} \to \pi_* f^* N_i,
  \end{equation*}
  where $f$ denotes the universal map, and $\pi$ is the projection as in the following diagram.
  \begin{equation*}
    \xymatrix{
      [J^{m_i - 1} Z_i / \Gm] \times (\Spec k[t]/(t^{m_i})) \ar[rr]^-f \ar[d]^\pi && Z_i \\
      [J^{m_i - 1} Z_i / \Gm]
    }
  \end{equation*}
  This establishes the assumptions for the application of the localization formula.
  To write the formula out, consider the restriction of the resolution to the fixed locus $Z_i$, that is
  \begin{equation*}
    \bigoplus_{k = 0}^{m_i - 1} TX|_{Z_i} \longto \bigoplus_{k = 0}^{m_i - 1} N_i,
  \end{equation*}
  where the $k$-th factor has $\Gm$-weight $-k$.
  The $k = 0$ summands form the fixed part of the perfect obstruction theory, which is the perfect obstruction theory of the fixed locus,
  while the remaining summands form the virtual normal bundle.
  Let $\rho\colon Z_i \to J^{m_i-1} Z_i$ and $\sigma\colon X_i \to J^{m_i-1} X_i$ denote the inclusions.
  The virtual localization formula yields the equality
  \begin{equation*}
    [J^{m_i-1} Z_i]^{\vir, \Gm}
    = \rho_* \left(\prod_{k = 1}^{m_i-1} \frac{e\big(N_i \otimes \Lambda^{\otimes (-k)}\big)}{e\big(T_X \otimes \Lambda^{\otimes (-k)}\big)} \cap [Z_i]^{\Gm} \right)
  \end{equation*}
  in the localization $A_*([J^{m_i-1} Z_i / \Gm])[t^{-1}]$.
  This implies
  \begin{align*}
  \tzeta_{i, *} [J^{m_i-1} Z_i]^{\vir, \Gm}
  &= \tzeta_{i, *} \rho_* \left(\prod_{k = 1}^{m_i-1} \frac{e\big(N_i \otimes \Lambda^{\otimes (-k)}\big)}{e\big(T_X \otimes \Lambda^{\otimes (-k)}\big)} \cap [Z_i]^{\Gm} \right)\\
  &= \sigma_* \iota_{i, *} \left(\prod_{k = 1}^{m_i-1} \frac{e\big(N_i \otimes \Lambda^{\otimes (-k)}\big)}{e\big(T_X \otimes \Lambda^{\otimes (-k)}\big)} \cap [Z_i]^{\Gm} \right)
  \end{align*}
  in $A_*([J^{m_i-1} X / \Gm])[t^{-1}]$.
  Applying the Gysin map and the excess intersection formula yields
  \begin{align*}
    \sigma^! \tzeta_{i, *} [J^{m_i-1} Z_i]^{\vir, \Gm}
    &= \sigma^! \sigma_* \iota_{i, *} \left(\prod_{k = 1}^{m_i-1} \frac{e\big(N_i \otimes \Lambda^{\otimes (-k)}\big)}{e\big(T_X \otimes \Lambda^{\otimes (-k)}\big)} \cap [Z_i]^{\Gm} \right) \\
    &= \iota_{i, *} \left(\prod_{k = 1}^{m_i-1} e\big(N_i \otimes \Lambda^{\otimes (-k)}\big) \cap [Z_i]^{\Gm} \right)
  \end{align*}
  in $A_*(X \times B\Gm)[t^{-1}]$.
  Since $A_*(X \times B\Gm) \simeq A_*(X)[t]$ and so $t$ is non-torsion, the same equality holds in $A_*(X \times B\Gm)$.
  Note that
  \begin{align*}
    (\tev_i^{m_i})^* \PD^{-1} \tzeta_{i, *} [J^{m_i-1} Z_i]^{\vir, \Gm}
    &= (\ev_i \times \varphi_i)^* \sigma^* \PD^{-1} \tzeta_{i, *} [J^{m_i-1} Z_i]^{\vir, \Gm} \\
    &= (\ev_i \times \varphi_i)^* \PD^{-1} \sigma^! \tzeta_{i, *} [J^{m_i-1} Z_i]^{\vir, \Gm},
  \end{align*}
  where in the first equality, we use that $\sigma^*$ is the inverse of the vertical arrow in Lemma~\ref{lem:cohomology-localizes}, and so $(\tev_i^{m_i})^* = (\ev_i \times \varphi_i)^* \sigma^*$.
  Combining the previous two equations, we get
  \begin{equation*}
    (\tev_i^{m_i})^* \PD^{-1} \tzeta_{i, *} [J^{m_i-1} Z_i]^{\vir, \Gm}
    = (\ev_i \times \varphi_i)^* \PD^{-1} \iota_{i, *} \left(\prod_{k = 1}^{m_i-1} e\big(N_i \otimes \Lambda^{\otimes (-k)}\big) \cap [Z_i]^{\Gm} \right) .
  \end{equation*}
  Thus, the first part of the proposition follows from \eqref{eq:tzeta-Gysin-PD}, and Lemma~\ref{lem:explicit-Gysin-pushforward} applied to $\iota_i\colon Z_i \to X$.
  The second part follows by a straightforward computation.
\end{proof}

\begin{corollary}
  The naive tangency condition
  \begin{equation*}
    (\ev_i^{m_i})^* \zeta_{i!}(1)
  \end{equation*}
  only depends on the cohomology class of $Z_i$, and not on the specific subvariety $Z_i$.
\end{corollary}

\begin{corollary}
  If $Z_i$ is the zero locus of a transverse section of a vector bundle $E_i$ of rank $r_i$ on $X$, the formula in Theorem~\ref{thm:naive_tangencies_to_descendents} may be further simplified to
  \begin{equation*}
    (\ev_i^{m_i})^* \zeta_{i!}(1)
    = (\ev_i \times \varphi_i)^* \prod_{k = 0}^{m_i - 1} e(E_i \boxtimes L_i^{\otimes -k})
    = \prod_{k = 0}^{m_i - 1} \sum_{j = 0}^{r_i} (k\psi_i)^{r_i - j} \ev_i^* c_j(E_i).
  \end{equation*}
\end{corollary}

\begin{corollary}
  \label{cor:divisor}
  As a special case of the previous corollary, if $Z_i$ is a Cartier divisor, and hence given as the zero section of a line bundle $L$,
  then
  \begin{equation}
    \label{eq:divisor}
    (\ev_i^{m_i})^* \zeta_{i!}(1)
    = \prod_{k = 0}^{m_i - 1} (k\psi_i + \ev_i^* c_1(L))
    = \sum_{k = 0}^{m_i} \genfrac[]{0pt}{1}{m_i}{k} \psi_i^{m_i - k} (\ev_i^* c_1(L))^k,
  \end{equation}
  where $\genfrac[]{0pt}{1}{m_i}{k}$ denotes the unsigned Stirling numbers of the first kind.
\end{corollary}

\begin{remark}
  The Gromov--Witten/Hurwitz correspondence of \cite{Okounkov_Gromov-Witten_theory_Hurwitz_theory_and_completed_cycles} for the Gromov--Witten theory of curves expresses a descendent point insertion as a linear combination of relative insertions, whose coefficients are \emph{completed cycles}.
  In a similar vein, we may recursively rewrite \eqref{eq:divisor} to express descendants in terms of naive tangency conditions:
  \begin{equation*}
    (m_i - 1)! \ev_i^*(c_1(L)) \psi_i^{m_i - 1}
    = \sum_{k = 1}^{m_i} c(m_i, k) (\ev_i^*(c_1(L)))^{m_i - k} (\ev_i^k)^* \zeta_{i!}(1)
  \end{equation*}
  for some constants $c(m, k) \in \QQ$ with $c(m, m) = 1$.
\end{remark}

\begin{example} \label{ex:curve}
  Let $X$ be a curve, and $Z_i$ a point.
  In this case, the $m_i$-th naive tangency condition translates to
  \begin{equation*}
    (\ev_i^{m_i})^* \zeta_{i!}(1)
    = (m_i - 1)! \psi_i^{m_i - 1} \ev_i^* \PD^{-1} [Z_i].
  \end{equation*}
  This recovers the insertion considered in \cite[Proposition~1.1]{Okounkov_Gromov-Witten_theory_Hurwitz_theory_and_completed_cycles}.
\end{example}
\begin{example} \label{ex:point}
  More generally, let $Z_i$ be a point in an $N$-fold $X$.
  Then, the naive tangency condition becomes
  \begin{equation*}
    (\ev_i^{m_i})^* \zeta_{i!}(1)
    = \big((m_i - 1)! \psi_i^{m_i - 1}\big)^N \ev_i^* \PD^{-1} [\pt].
  \end{equation*}
\end{example}

\section{Examples and mirror symmetry} \label{sec:examples}

In this section, we compute several examples of Gromov--Witten invariants with naive tangency conditions.
We start with the case of curves, where we relate Gromov--Witten invariants with naive tangencies to Hurwitz numbers with completed cycles.
Next, we consider the case of surfaces.
We first compute several examples where we compare the invariants to enumerative counts, concluding with Gathmann's example of conics tangent to a plane curve in $5$ points.
After this, we study the counts of rational curves naively tangent to an anticanonical divisor on a del Pezzo surface, and relate them to the local Gromov--Witten invariants via mirror symmetry.

\subsection{Curves and Hurwitz numbers} \label{sec:curves}

Let $X$ be a smooth projective curve with $n$ points $q_1, \dots, q_n$.
Let $d$, $k_1, \dots, k_n$ be positive integers.
Denote by
\[ \bigg\langle\prod_{i=1}^n k_i q_i\bigg\rangle_d^{\bullet X} \]
the disconnected Gromov--Witten invariants of degree $d$, with naive tangencies at point $q_i$ with order $k_i$, and with genus determined by the virtual dimension constraint.
It follows from Example \ref{ex:curve} and \cite[(0.24)]{Okounkov_Gromov-Witten_theory_Hurwitz_theory_and_completed_cycles} that
\[ \bigg\langle\prod_{i=1}^n k_i q_i\bigg\rangle_d^{\bullet X} = \mathsf H_d^X\Big(\overline{(k_1)},\dots,\overline{(k_n)}
\Big), \]
where the right-hand side denotes the disconnected Hurwitz numbers with completed cycles.
Such Hurwitz numbers can be computed explicitly as follows.
For a partition $\lambda = (\lambda_1, \lambda_2, \dots)$ of $d$, let $\mathbf p_k(\lambda)$ denote the coefficient of $z^k$ in the Laurent series expansion of
\[k!\sum_{i=0}^\infty e^{z(\lambda_i-i+\frac{1}{2})},\]
see \cite[\S0.4.3]{Okounkov_Gromov-Witten_theory_Hurwitz_theory_and_completed_cycles} for details on how to interpret this infinite sum.
Then by \cite[(0.25)]{Okounkov_Gromov-Witten_theory_Hurwitz_theory_and_completed_cycles}, we obtain
\[ \bigg\langle\prod_{i=1}^n k_i q_i\bigg\rangle_d^{\bullet X} = \sum_{\abs{\lambda}=d} \bigg(\frac{\dim\lambda}{d!}\bigg)^{2-2g(X)} \prod_{i=1}^n \frac{\mathbf p_{k_i}(\lambda)}{k_i}. \]

\bigskip
Next, we study Gromov--Witten invariants with naive tangencies in the case of surfaces.

\subsection{Lines tangent to a conic} \label{sec:lines_tangent_to_conic}

We consider the Gromov--Witten invariant with naive tangencies corresponding to lines in $X = \PP^2$ through a point $p \in X$ and tangent to a disjoint conic $Z_2 = C$:
\begin{equation*}
  \langle p, 2 C\rangle_{0, 1}^X = \int_{[\bcM_{0, 2}(X, 1)]^\vir} \ev_1^*(\PD^{-1} [p]) (\ev_2^2)^* \zeta_{2!}(1).
\end{equation*}

To compute this number, we first rewrite the invariant using Corrolary~\ref{cor:divisor}
\[
  \int_{[\bcM_{0, 2}(X, 1)]^\vir} \ev_1^*(H^2) \ev_2^*(2H) (\psi_2 + \ev_2^*(2H))
  = 2\langle \tau_1(H) \tau_0(H^2) \rangle_{0, 1}^X + 4\langle \tau_0(H^2) \tau_0(H^2)\rangle_{0, 1}^X.
\]
Here, $H$ denotes the hyperplane class, and we use the bracket notation
\begin{equation*}
  \langle \tau_{a_1}(\alpha_1) \cdots \tau_{a_n}(\alpha_n)\rangle_{g, d}^X
  = \int_{[\bcM_{g, n}(X, d)]^\vir} \prod_{i = 1}^n \psi_i^{a_i} \ev_i^*(\alpha_i).
\end{equation*}

We compute the second summand via the well-known Gromov--Witten invariant
\begin{equation*}
  \langle \tau_0(H^2) \tau_0(H^2)\rangle_{0, 1}^X = 1,
\end{equation*}
which is equal to the number of lines through two points.
To compute the first summand, we combine an application of the divisor equation
\begin{equation*}
  \langle \tau_1(H) \tau_0(H^2) \tau_0(H)\rangle_{0, 1}^X
  = \langle \tau_1(H) \tau_0(H^2)\rangle_{0, 1}^X + \langle \tau_0(H^2) \tau_0(H^2)\rangle_{0, 1}^X
\end{equation*}
and of the genus zero topological recursion relation
\begin{equation*}
  \langle \tau_1(H) \tau_0(H^2) \tau_0(H)\rangle_{0, 1}^X
  = \sum_{\alpha = 0}^2 \langle \tau_0(H^2) \tau_0(H) \tau_0(H^\alpha)\rangle_{0, 0}^X \langle \tau_0(H^{2 - \alpha}) \tau_0(H)\rangle_{0, 1}^X
  = 0.
\end{equation*}
Therefore, we have
\begin{equation*}
  \langle \tau_1(H) \tau_0(H^2)\rangle_{0, 1}^X = -1,
\end{equation*}
and we conclude
\begin{equation*}
  \langle p, 2 C\rangle_{0, 1}^X = 2.
\end{equation*}

\begin{remark}
  This agrees with the enumerative fact that there are exactly two lines passing through $p$ and tangent to $C$.
\end{remark}

\begin{remark}
  We could repeat the same computation with $C$ replaced by a plane curve of degree $d$.
  In this more general case, we have
  \begin{equation*}
    \langle p, 2 C\rangle_{0, 1}^X = d(d-1).
  \end{equation*}
  In particular, this vanishes if $d = 1$, which is in agreement with the geometric fact that the maximal tangency order between two lines is $1$.
  In genus $g \ge 2$, there is no topological recursion relation for $\psi_2$, and thus we do not expect Gromov--Witten invariants to vanish when there is an insertion with naive tangency of order exceeding $d$.
\end{remark}

\subsection{Conics tangent to a line} \label{sec:conics_tangent_to_line}

We now consider the naive Gromov--Witten invariant corresponding to counting conics in $X = \PP^2$ through $4$ points $Z_1$, $Z_2$, $Z_3$,
$Z_4$, and tangent to a line $Z_5$:
\begin{align*}
  &\langle Z_1, Z_2, Z_3, Z_4, 2 Z_5\rangle_{0, 2}^X \\
  = &\int_{[\bcM_{0, 5}(X, 2)]^\vir} \prod_{i = 1}^4 \ev_i^*(H^2) \cdot \ev_5^*(H) \cdot (\psi_5 + \ev_5^*(H)) \\
  = &\langle \tau_1(H) \tau_0(H^2) \tau_0(H^2) \tau_0(H^2) \tau_0(H^2) \rangle_{0, 2} \\
  & + \langle \tau_0(H^2) \tau_0(H^2) \tau_0(H^2) \tau_0(H^2) \tau_0(H^2)\rangle_{0, 2},
\end{align*}
where we again used Corollary~\ref{cor:divisor}.
For the second term, note that
\begin{equation*}
  \langle \tau_0(H^2) \tau_0(H^2) \tau_0(H^2) \tau_0(H^2) \tau_0(H^2)\rangle_{0, 2} = 1.
\end{equation*}
We may compute the first term via the topological recursion relations:
\begin{multline*}
  \langle \tau_1(H) \tau_0(H^2) \tau_0(H^2) \tau_0(H^2) \tau_0(H^2) \rangle_{0, 2} \\
  = \langle \tau_0(H^2) \tau_0(H^2) \tau_0(H) \rangle_{0, 1} \langle \tau_0(H) \tau_0(H) \tau_0(H^2) \tau_0(H^2) \rangle_{0, 1} = 1.
\end{multline*}
So, we conclude
\begin{equation*}
  \langle Z_1, Z_2, Z_3, Z_4, 2 Z_5\rangle_{0, 2}^X = 2.
\end{equation*}
\begin{remark}
  This is in agreement with the fact that there are two conics through $4$ general points and tangent to a general line.
\end{remark}

\subsection{Conics with multiple tangencies}

We now give an example of a Gromov--Witten invariant with multiple incidence conditions at one point, namely, we consider the invariant corresponding to conics in $X = \PP^2$ passing through three points $Z_1$, $Z_2$, $Z_3$ at three marked points, and tangent to a line $Z_4$ and passing through another line $Z_5$ at the fourth marked point:
\begin{align*}
  &\langle Z_1, Z_2, Z_3, 2Z_4 \cap Z_5\rangle_{0, 2}^X \\
  = &\int_{[\bcM_{0, 4}(X, 2)]^\vir} \prod_{i = 1}^3 \ev_i^*(H^2) \cdot \ev_4^*(H) \cdot (\ev_4^2)^* \zeta_{4!}(1) \\
  = &\int_{[\bcM_{0, 4}(X, 2)]^\vir} \prod_{i = 1}^3 \ev_i^*(H^2) \cdot \ev_4^*(H) \cdot \ev_4^*(H) \cdot (\psi_4 + \ev_4^*(H)) \\
  = &\langle \tau_1(H^2) \tau_0(H^2) \tau_0(H^2) \tau_0(H^2) \rangle_{0, 2} .
\end{align*}
By the topological recursion relations,
\[
  \langle Z_1, Z_2, Z_3, 2Z_4 \cap Z_5\rangle_{0, 2}^X
  = \langle \tau_0(H^2) \tau_0(H^2) \tau_0(H) \rangle_{0, 1} \langle \tau_0(H) \tau_0(H^2) \tau_0(H^2) \rangle_{0, 1} = 1 \cdot 1 = 1.
\]
\begin{remark}
  This is also in agreement with the corresponding enumerative count.
  After imposing the conics to pass through three general points and the intersection of the two general lines, we are left with a $\PP^1$-family of conics, and imposing the tangency along $Z_5$ is then only a linear condition.
\end{remark}

\subsection{Conics tangent to a plane curve}
\label{sec:gathmann}

In \cite{Ga05}, Gathmann considered the enumerative count $n_d$ of plane conics that are simply tangent to a fixed plane curve $C$ of degree $d$ in $5$ points.
This count is of interest due to a connection to counts on K3 surfaces given as a double cover of $\PP^2$ branched along a sextic.
Gathmann finds
\begin{equation*}
  n_d = \frac 1{5!} d(d - 3)(d - 4)(d^7 + 12 d^6 - 18d^5 - 540d^4 + 251 d^3 + 5712 d^2 - 1458 d - 14580).
\end{equation*}
Let us compute the corresponding Gromov--Witten invariants with naive tangencies
\begin{equation*}
  n_d^{\GW} = \frac 1{5!} \langle 2 C, 2 C, 2 C, 2 C, 2 C\rangle_{0, 2}.
\end{equation*}

By Corollary~\ref{cor:divisor},
\begin{align*}
  5! n_d^{\GW}
  &= \int_{[\bcM_{0, 5}(X, 2)]^\vir} \prod_{i = 1}^5 \ev_i^*(dH) \cdot (\psi_i + \ev_i^*(dH)) \\
  &= \sum_{i = 0}^5 \binom{5}{i} d^{5 + i} \langle (\tau_0(H^2))^i (\tau_1(H))^{5-i} \rangle_{0, 2}.
\end{align*}
By the computations in Section~\ref{sec:conics_tangent_to_line} and using the axioms of Gromov--Witten invariants, we find
\begin{equation*}
  n_d^{\GW} = \frac 1{5!} d^8 (d^2 + 5d + 20).
\end{equation*}
Note that both $n_d$ and $n_d^{\GW}$ are degree $10$ polynomials in $d$ whose leading and first subleading coefficients agree.
Nevertheless, the computation shows that Gromov--Witten invariants with naive tangencies can be far from the enumerative counts in this example.

\subsection{Curves tangent to an anti-canonical divisor on a del Pezzo surface}
\label{sec:arbitrary-del-Pezzo}

In \cite{GGR19}, van Garrel, Graber and Ruddat established a relation between log Gromov--Witten invariants with maximal tangency along a smooth divisor and ordinary Gromov--Witten invariants of a local geometry.
As a special case, the relative Gromov--Witten theory of $\PP^2$ relative to an elliptic curve was studied by Gathmann in \cite[Example~2.2]{Ga03}, and related to local mirror symmetry.
In this section, we establish a relationship between Gromov--Witten invariants with naive tangencies and local Gromov--Witten invariants, in the case of del Pezzo surfaces.

Let $X$ be a del Pezzo surface of degree $d_X$, and $D$ be an anti-canonical divisor.
We consider the Gromov--Witten invariants with maximal naive tangency at one point along $D$:
\begin{equation*}
  \langle (\beta \cdot D) D\rangle_{0, \beta}^X = \int_{[\bcM_{0, 1}(X, \beta)]^\vir} \left(\ev_1^{\beta \cdot D}\right)^* \zeta_{1!}(1) .
\end{equation*}
We will relate these invariants to local Gromov--Witten invariants of $X$, which uses local mirror symmetry in a crucial way.

Let $\Lambda$ be the completion of $\QQ[\NE(X, \ZZ)]$ along its maximal monomial ideal, with basis $\{q^\beta: \beta \in \NE(X, \ZZ)\}$.
We define a generating series $F(q) \in \Lambda$ of Gromov--Witten invariants with naive tangencies via
\begin{equation*}
  F(q) = \sum_\beta q^\beta \langle (\beta \cdot D) D\rangle_{0, \beta}^X
  = \sum_\beta q^\beta \int_{[\bcM_{0, 1}(X, \beta)]^\vir} \prod_{k = 0}^{\beta \cdot D - 1} (k\psi_1 + \ev_1^*(D)),
\end{equation*}
where we sum over all effective curve classes $\beta \in \NE(X, \ZZ)$.
Here, and in what follows, we will follow the convention that integrals over the empty moduli space $\bcM_{0, 1}(X, 0)$ are defined to be zero.

We begin by expressing $F(q)$ in terms of the small $J$-function of $X$
\begin{equation*}
  J_X = \sum_\beta q^\beta J_X^\beta
  = z + \sum_\beta q^\beta \PD^{-1} \ev_{1*} \frac{[\bcM_{0, 1}(X, d)]^\vir}{z - \psi_1} ,
\end{equation*}
which is viewed as an element of $zH^*(X) \otimes \Lambda[\![z^{-1}]\!]$.

\begin{lemma}
  We have the identity
  \begin{equation*}
    F(q) = \Res_{z = 0} \int_X I_{\KX}(q)|_{q^\beta \mapsto (-1)^{\beta \cdot D} q^\beta},
  \end{equation*}
  where
  \begin{equation*}
    I_{\KX}(q) = \sum_\beta q^\beta J_X^\beta(z) \prod_{k = 0}^{\beta \cdot D - 1}(-D - kz)
  \end{equation*}
  is the ``hypergeometric modification'' (in the sense of \cite[\S7]{CoGi07}) of $J_X$.
\end{lemma}
\begin{proof}
  This follows from
  \begin{align*}
    F(q) &= \Res_{z = 0} \sum_\beta q^\beta \int_{[\bcM_{0, 1}(X, \beta)]^\vir} \frac 1{z - \psi_1} \prod_{k = 0}^{\beta \cdot D - 1} (\ev_1^*(D) + k\psi_1) \\
         &= \Res_{z = 0} \sum_\beta q^\beta \int_{[\bcM_{0, 1}(X, \beta)]^\vir} \frac 1{z - \psi_1} \prod_{k = 0}^{\beta \cdot D - 1} (\ev_1^*(D) + kz) \\
         &= \Res_{z = 0} \sum_\beta q^\beta \int_X J_X^\beta \prod_{k = 0}^{\beta \cdot D - 1} (D + kz),
  \end{align*}
  where in the second equality, we used that $(z - \psi) | (f(\psi) - f(z))$ for any power series $f$.
  Note that $J_X^0 = z$ does not contribute to the residue.
\end{proof}

\begin{lemma}
  There is an expansion
  \begin{equation*}
    I_{\KX}(q) = z + I_1(q) D + I_2(q) \frac{D^2}z,
  \end{equation*}
  with $I_1(q), I_2(q) \in \Lambda$.
\end{lemma}
\begin{proof}
  By virtual dimension considerations the series
  $J_X^\beta(z)$
  is homogeneous of degree $2 - 2(\beta \cdot D)$ under the convention that $\alpha \cdot z^b$ with $\alpha \in H^d(X)$ has
  degree $d + 2b$.
  Therefore, $I_{\KX}(q)$ is homogeneous of degree $2$ if we further
  set the Novikov ring to be homogeneous of degree zero.
  Thus, there is an expansion
  \begin{equation*}
    I_{\KX}(q) = z + \mathbf I_1(q) + \mathbf I_2(q) \frac 1z + O(z^{-2}),
  \end{equation*}
  with $\mathbf I_1(q) \in H^2(X) \otimes \Lambda$ and $\mathbf I_2(q) \in H^4(X) \otimes \Lambda$.
  Furthermore, note that all except the $\beta = 0$ summand of $I_\KX(q)$ are divisible by $D$, so that $\mathbf I_1(q)$ must be a multiple of $D$.
  Since $X$ is a surface, $\mathbf I_2(q)$ is a multiple of $D^2$, and $I_{\KX}(q)$ has no terms past the $z^{-1}$-term.
\end{proof}

\begin{corollary}
  \label{cor:naive-vs-loc}
  We have the explicit formula
  \begin{equation*}
    F(q) = d_X I_2(q)|_{q^\beta \mapsto (-1)^{\beta \cdot D} q^\beta}.
  \end{equation*}
\end{corollary}

\begin{example}
  \label{ex:local-P2}
  If $X = \PP^2$, by toric mirror symmetry (see \cite{Gi98b}), we have
  \begin{equation*}
    J_X = I_X = z\sum_{\beta = 0}^\infty q^\beta \frac 1{\prod_{k = 1}^\beta (H + iz)^3},
  \end{equation*}
  where $H$ denotes the hyperplane class.
  Therefore,
  \begin{equation*}
    I_{\KX} = z\sum_{\beta = 0}^\infty q^\beta \frac{\prod_{k = 0}^{3\beta - 1} (-3H - kz)}{\prod_{k = 1}^\beta (H + kz)^3},
  \end{equation*}
  and
  \begin{align*}
    I_1(q) &= \sum_{\beta = 0}^\infty (-q)^\beta \frac{(3\beta - 1)!}{\beta!^3}, \\
    I_2(q) &= \frac 13 \sum_{\beta = 0}^\infty (-q)^\beta \frac{(3\beta - 1)!}{\beta!^3} \sum_{k = \beta + 1}^{3\beta - 1} \frac 1k.
  \end{align*}
\end{example}

We next compare $F(q)$ to the generating function
\begin{equation*}
  F^\loc(q) = \sum_\beta q^\beta \int_{[\bcM_{0, 1}(\KX, \beta)]^\vir} \ev_1^*(D) \ \in\ \Lambda
\end{equation*}
of genus zero Gromov--Witten invariants of local $X$.
\begin{proposition}
  \label{prop:compute-Floc}
  We have the explicit formula
  \begin{align*}
    -F^\loc(Q) = d_X \left(I_2(q) - \frac 12 I_1^2(q)\right).
  \end{align*}
  under the variable change (mirror map)
  \begin{equation}
    \label{eq:mirror-map}
    Q^\beta = q^\beta \cdot \exp((\beta \cdot D) I_1(q)).
  \end{equation}
\end{proposition}
\begin{proof}
  Consider the $J$-function
  \begin{equation*}
    J_{\KX}(t, q) = z + t - D \sum_\beta \sum_{k = 0}^\infty \frac{q^\beta}{k!} \PD^{-1} \ev_{1*} \left(\prod_{i = 2}^{k + 1} \ev_i^*(t) \cap \frac{[\bcM_{0, 1 + k}(\KX, \beta)]^\vir}{z - \psi_1}\right),
  \end{equation*}
  where $t \in H^*(X)$ is a parameter.
  Note that
  \begin{equation*}
    F^\loc(q) = -\Res_{z = 0} \int_X J_{\KX}(0, q),
  \end{equation*}
  and hence
  \begin{equation*}
    F^\loc(Q) = -\Res_{z = 0} \int_X J_{\KX}(0, Q).
  \end{equation*}

  By \cite[\S9 and \S10 (on Serre duality)]{CoGi07}, the hypergeometric modification $I_{\KX}(q)$ lies on Givental's Lagrangian cone for the Gromov--Witten theory of local $X$.
  Therefore,
  \begin{equation*}
    I_{\KX}(q) = J_{\KX}(I_1(q) D, q).
  \end{equation*}
  By repeated application of the divisor equation (see \cite[\S8]{CoGi07}), we find
  \begin{equation*}
    I_{\KX}(q) = e^{I_1(q) \frac Dz} J_{\KX}(0, Q)
  \end{equation*}
  with respect to the mirror map \eqref{eq:mirror-map}.
  
  We conclude
  \begin{align*}
    F^\loc(Q)
    =& - \Res_{z = 0} \int_X J_{\KX}(0, Q) \\
    =& - \Res_{z = 0} \int_X e^{-I_1(q) \frac Dz} I_{\KX}(q) \\
    =& -d_X\left(I_2(q) - \frac 12 I_1^2(q)\right).
  \end{align*}
\end{proof}
\begin{theorem}
  \label{prop:naive-to-log}
  We have the identity
  \begin{equation*}
    F(q)|_{q^\beta \mapsto (-1)^{\beta \cdot D} q^\beta} - \frac{d_X}2 I_1^2(q)= -F^\loc(Q)
  \end{equation*}
  under the mirror map \eqref{eq:mirror-map}.
\end{theorem}
\begin{proof}
  This follows by combining Corollary~\ref{cor:naive-vs-loc} and Proposition~\ref{prop:compute-Floc}.
\end{proof}

Let us now compare to the result \cite{GGR19}.
The authors use slightly different local Gromov--Witten invariants
\begin{equation*}
  \deg [\bcM_{0, 0}(\KX, \beta)]^\vir.
\end{equation*}
The divisor equation connects these to the local Gromov--Witten
invariants that define $F^\loc$:
\begin{equation*}
  \int_{[\bcM_{0, 1}(\KX, \beta)]^\vir} \ev_1^*(D) = (\beta \cdot D) \deg [\bcM_{0, 0}(\KX, \beta)]^\vir.
\end{equation*}
Furthermore, we form the generating series
\begin{equation*}
  F^{\log}(q) = \sum_\beta q^\beta n_{\beta}^{\log}
\end{equation*}
of log Gromov--Witten invariants $n_{\beta}^{\log}$ of $X$ with maximal contact along $D$ at a single marked point.
It follows from \cite[Theorem~1.1]{GGR19} that
\begin{equation*}
  F^{\log}(Q)|_{Q^\beta \mapsto (-1)^{\beta \cdot D} Q^\beta} = - F^{\loc}(Q) .
\end{equation*}

Therefore, we obtain the following corollary.

\begin{corollary}
  \label{cor:naive-to-log}
  We have the identity
  \begin{equation*}
    F(q)|_{q^\beta \mapsto (-1)^{\beta \cdot D} q^\beta} - \frac{d_X}2 I_1^2(q) = F^{\log}(Q)|_{Q^\beta \mapsto (-1)^{\beta \cdot D} Q^\beta}
  \end{equation*}
  between $F$ and $F^{\log}$ under the mirror transformation \eqref{eq:mirror-map}.
\end{corollary}

\bibliographystyle{plain}
\bibliography{biblio}

\end{document}